\documentclass{article}
\usepackage{graphicx} 
\usepackage{amsfonts,epsf,amsmath,amssymb,amsthm}
\usepackage{epsfig}
\usepackage{latexsym}
\usepackage{tikz}
\usetikzlibrary{decorations.markings}

\usepackage{tkz-graph}
\usetikzlibrary{positioning}
\usetikzlibrary{shapes.geometric}
\usetikzlibrary{positioning}
\usepackage[colorlinks=false,urlbordercolor=white]{hyperref}
  \def\gn#1#2{{$\href{http://groupnames.org/\#?#1}{#2}$}}
\def\gn#1#2{$#2$}  

\tikzstyle{vertex}=[circle,fill, draw, inner sep=0pt, minimum size=6pt]
\usepackage[letterpaper,top=2cm,bottom=2cm,left=3cm,right=3cm,marginparwidth=1.75cm]{geometry}

\newtheorem{theorem}{\bf Theorem}[section]
\newtheorem{corollary}[theorem]{\bf Corollary}
\newtheorem{lemma}[theorem]{\bf Lemma}
\newtheorem{example}[theorem]{\bf Example}

\newtheorem{remark}[theorem]{\bf Remark}
\newtheorem{definition}[theorem]{\bf Definition}

\begin{document}

\title{Cyclic Subgroup Lattices as Universal Sources of Power-Type Graphs}

\let\cleardoublepage\clearpage
\author{M. Mirzargar$^{a,b}$, S. Sorgun$^{a}$, M. J. Nadjafi-Arani$^{b,*}$ \\
	\small $^{a}$Department of Mathematics, Nevşehir Hacı Bektaş Veli University, 50300, Nevşehir, Turkey \\
	\small $^{b}$
Faculty of Science, Mahallat Institute of Higher Education, Mahallat, Iran \\
\small $^{*}$Corresponding author \tt{mjnajafiarani@gmail.com; mjnajafiarani@mahallat.ac.ir} \\
	\small \tt{m.mirzargar@mahallat.ac.ir} \\
    \small \tt{srgnrzs@gmail.com; ssorgun@nevsehir.edu.tr } \\
}	
	\maketitle
	
	\begin{abstract}
		\noindent
Power-type graphs, such as the power graph, the directed power graph, the enhanced power graph and the difference graph, encode significant information about the internal structure of a finite group.
Despite substantial investigation in recent years, the precise relationship between these graphs and the subgroup lattice of the underlying group has remained only partially understood. In this paper we establish a complete, explicit, and purely combinatorial correspondence between the enhanced power graph and the lattice of cyclic subgroups $\mathcal{L}_c(G)$ of a finite group $G$. We prove that these two objects determine each other uniquely:
an unlabeled enhanced power graph suffices to reconstruct $\mathcal{L}_c(G)$, and conversely, the labeled enhanced power graph can be reconstructed directly from $\mathcal{L}_c(G)$.
Exploiting this duality, we demonstrate that the reconstruction principle applies equally to the power graph, the directed power graph, and the difference graph, which may all be derived solely from the cyclic subgroup lattice, independent of the group operation.
 This bidirectional correspondence yields a combinatorial equivalence between power-type graphs and the cyclic subgroup lattice, providing a new framework for analyzing finite groups through graph-theoretic and lattice-theoretic data, free from algebraic complexity. \\
        

    
\end{abstract}

\maketitle

\section{Introduction}

Several graphs have been introduced on the set of elements of a group to capture various aspects of its algebraic structure in graph theoretic terms. 
Among these, power graph and enhanced power graph have emerged as one of the most effective tools for revealing the internal organization of a group. 
 Studying a group $G$ through its associated graph 
highlights many structural properties of $G$, while translating the algebraic complexity of the group into a more tractable combinatorial framework. The fundamental problem motivating such constructions is to what extent the graph associated with a group $G$ encodes or determines the intrinsic algebraic structure of $G$ itself.
Among these graphs, the enhanced power graph plays a particularly revealing role, as it encodes detailed information about the cyclic subgroups of $G$ and the way they intersect. Moreover, each maximal clique in the enhanced power graph corresponds to a maximal cyclic subgroup of $G$.

Denote $\mathcal{L}_c(G)$, \textit{the lattice of  cyclic subgroups of $G$} , ordered by inclusion; throughout the paper, we identify $\mathcal{L}_c(G)$ with its Hasse diagram, where the vertices represent cyclic subgroups and the edges correspond to the cover relations under inclusion. Understanding how the enhanced power graph  determines 
$\mathcal{L}_c(G)$ is a key step toward recovering deeper structural properties of $G$ from its graph representation, and it  motivates our main goal, to show that the enhanced power graph and the lattice of cyclic subgroups are, in essence, equivalent invariants for  finite groups.


 In this paper, we deal only with finite groups, and we call a graph $\Gamma$ an enhanced power graph if there exists at least one finite group $G$ such that $\Gamma\sim \mathrm{EPow}(G) $. Our Main Theorem establishes that, given an unlabeled enhanced power graph $\Gamma$, for a certain finite group $G$, the lattice of cyclic subgroups $\mathcal{L}_c(G)$  can be reconstructed purely from arithmetical and graph theoretical considerations, without any prior knowledge of the group structure itself.
Conversely, starting from a given cyclic subgroup lattice $\mathcal{L}_c(G)$ of a finite group $G$,   one can uniquely reconstruct the labeled enhanced power graph $\mathrm{EPow}(G)$. However, the  main theorem  of this paper is as follows:\\

\textbf{Main Theorem.} Let $G$ be a finite group. 
Then the enhanced power graph and the cyclic subgroup lattice of $G$ determine each other uniquely. 
More precisely, the cyclic subgroup lattice $\mathcal{L}_c(G)$ can be reconstructed from any unlabeled enhanced power graph $\mathrm{EPow(G)}$. Conversely, labeled enhanced power graph $G$ can be reconstructed from  $\mathcal{L}_c(G)$. 
Moreover, the labeled power graph, the directed power graph, and the difference graph of $G$ can also be reconstructed from  $\mathcal{L}_c(G)$.
\vspace{4mm}

Once the labeled enhanced power graph has been reconstructed uniquely from the cyclic subgroup lattice, 
the same reasoning extends naturally to the power graph. 
Since the enhanced power graph differs from the power graph only by the additional edges joining generators of the same cyclic subgroup, 
the reconstruction procedure applies in an analogous way. 
Moreover, by introducing the natural orientation of edges induced by subgroup inclusion, 
the same data in $\mathcal{L}_c(G)$ also determines the directed power graph. 
Consequently, the difference graph  can likewise be derived from $\mathcal{L}_c(G)$. 
Thus, all power-type graphs  completely determined by its cyclic subgroup lattice.

Moreover, when the enhanced power graph and the cyclic subgroup lattice of $G$ determine each other uniquely, the class of groups characterized by their enhanced power graphs coincides with those characterized by their cyclic subgroup lattices.
This correspondence naturally extends the relationship between graphs and lattices from the cyclic subgroup lattice to the full subgroup lattice, thereby generalizing the isomorphism problem to the entire subgroup structure of $G$. The same conclusion in the isomorphism problem also applies to the class of groups characterized by their  power-type graphs.

Section 2 presents the related work and preliminaries, including the fundamental definitions of the groups and graphs used throughout the paper, as well as a brief overview of the relevant literature. Section 3 states our main results, showing that the enhanced power graph and the cyclic subgroup lattice of a group determine each other uniquely.
Section 4 describes the reconstruction of the power graph, directed power graph, and difference graph from $\mathcal{L}_c(G)$.
 Section 5 addresses the isomorphism problem, determining which groups are uniquely characterized by their enhanced power graphs or, equivalently, by their subgroup lattices. In the final section, we outline several natural directions for further research arising from the reconstruction framework developed in this paper. 

\section{Related Work and Preliminaries}

A \emph{clique} in a graph $\Gamma$ is a subset of vertices in which every pair of distinct vertices is adjacent.    Equivalently, a clique is an induced complete subgraph of $\Gamma$. A clique $K$ in $\Gamma$ is called \emph{maximal clique} if it is not properly contained in any larger clique.  
A cyclic subgroup $C \le G$ is called \emph{maximal cyclic subgroup} if it is not properly contained in any larger cyclic subgroup of $G$.  
Equivalently, $C$ is maximal cyclic if whenever $C \subseteq H \le G$ and $H$ is cyclic, then $H = C$. Let $\varphi(d)$ denote \emph{Euler’s totient function}, i.e., the number of positive integers less than or equal to $d$ that are coprime to $d$. It is well-known that the generators of a cyclic subgroup of order $d$ are equal to $\varphi(d).$

The \emph{directed power graph} of a group \( G \), denoted \( \overrightarrow{\mathrm{Pow}}(G) \), 
is the directed graph whose vertex set is \( G \), with an arc \( x \to y \) whenever \( y = x^m \) for some integer \( m \). The directed power graph was introduced by Kelarev and Quinn, \cite{kelarev2000combinatorial}. 
The \emph{power graph} of \( G \), denoted \( \mathrm{Pow}(G) \), is obtained by ignoring directions and multiple arcs; 
equivalently, two distinct elements \( x, y \in G \) are adjacent if one is a power of the other. 
The notion of the power graph was first studied in~\cite{chakrabarty2009undirected}, and a comprehensive survey of its development and main results is provided in~\cite{kumar2021recent}.

The \emph{enhanced power graph} \( \mathrm{EPow}(G) \) of a group \( G \) is the simple undirected graph whose vertex set is \( G \), 
where two distinct elements \( x, y \in G \) are adjacent if and only if they generate a common cyclic subgroup; 
that is, \( \langle x \rangle = \langle y \rangle \) or \( \langle x, y \rangle \) is cyclic. 
The term enhanced power graph was introduced by Aalipour et al.~\cite{aalipour2016structure} 
as a graph “lying between” the power graph and the commuting graph, (which commuting graph has vertex set \( G \), 
with two distinct elements \( x, y \) adjacent whenever \( xy = yx \)).  A key fact used throughout this paper is that every maximal clique in the enhanced power graph 
of a finite group \( G \) corresponds to a maximal cyclic subgroup of \( G \)~\cite{aalipour2016structure}. 
Conversely, each cyclic subgroup of \( G \) induces a clique in \( \mathrm{EPow}(G) \). It is clear that, for any group \( G \),
$E(\mathrm{Pow}(G)) \subseteq E(\mathrm{EPow}(G))$,
where \( E(\Gamma) \) denotes the edge set of a graph \( \Gamma \). 
For further studies on enhanced power graphs, see~\cite{ma2022survey}. 

The \emph{difference graph} of a group \( G \), denoted \( \mathrm{D}(G) \),  
is defined as the graph whose edge set is
  $ E(\mathrm{D}(G)) \;=\; E(\mathrm{EPow}(G)) \setminus E(\mathrm{Pow}(G))$,
with all isolated vertices removed. For more details on the difference graph, see~\cite{biswas2024difference}.

In~\cite{bubboloni2025critical, das2023isomorphismproblempowergraphs}, the authors present an algorithm for reconstructing the directed power graph of a finite group 
from its undirected counterpart, first posed by Cameron as a question in~\cite{cameron2022}. Their method relies purely on arithmetical and graph-theoretical considerations, 
without using any group-theoretical information about \( G \). 
This approach is conceptually related to ours: while their reconstruction begins with the power graph, 
our method starts from the cyclic subgroup lattice \( \mathcal{L}_c(G) \) and produces all power-type graphs. This yields a unified framework that links the subgroup lattice to the entire family of power-type graphs, showing that each of these graphical representations is determined purely by the lattice’s inherent combinatorial properties. As a consequence, questions concerning groups determined by their power graphs (directed or enhanced), as studied in~\cite{mirzargar2022finite,mirzargar2025finite}, reduce naturally to questions about groups determined by their subgroup lattices.
Moreover, if two groups have isomorphic subgroup lattices, as investigated in~\cite{tuarnuauceanu2006groups}, then their power graphs (directed or enhanced) are all isomorphic as well.

\section{Main Results}
Let $\mathcal{C}_G$ denote the set of all cyclic subgroups of $G$. 
For each $M \in \mathcal{C}_G$ with $|M| = n$, consider the (internal) chain of all subgroups of $M$ ordered by inclusion:
\[
\{e\} = M_{d_1} < M_{d_2} < \cdots < M_{d_k} = M,
\]
where $1 = d_1 < d_2 < \cdots < d_k = n$ and $d_i$'s are positive divisors of $n$. Furthermore,
$M_{d_i}$ denotes the unique subgroup of $M$ of order $d_i$ for each $d_i \mid n$.

\begin{lemma}\label{lem1}
If $M, N \in \mathcal{C}_G$ and $d = |M \cap N|$,  
then for every divisor $m$ of $d$, the unique subgroups of order $m$ in $M$ and $N$ coincide with each other.

\end{lemma}

\begin{proof}
Let the chains of subgroups of $M$ and $N$ containing $M \cap N$ and the unique subgroup of order $m$ as follows:
\[
\{e\} = M_{m_1} < M_{m_2} < \cdots < M_{m_r} = M,
\qquad
\{e\} = N_{n_1} < N_{n_2} < \cdots < N_{n_s} = N,
\]
where $1 = m_1 < \cdots < m_r = |M|$ and 
$1 = n_1 < \cdots < n_s = |N|$ are the divisors of $|M|$ and $|N|$, respectively. For each divisor $m$ of $d$, the cyclic group $M \cap N$ has a unique subgroup of order $m$; denote it by $H_m$.  
Clearly, $H_m \subseteq M \cap N \subseteq M, N$.
  Therefore, the unique subgroups of order $m$ in $M$ and $N$ coincide.
\end{proof}

\begin{theorem}\label{th1}
The cyclic subgroup lattice  $\mathcal{L}_c(G)$  of a group $G$  can be reconstructed uniquely from the adjacency structure of $\mathrm{EPow}(G)$, without reference to the elements of $G$.

\end{theorem}

\begin{proof}
List all maximal cliques of $\mathrm{EPow}(G)$. For each maximal clique $K_C$ corresponding to a maximal cyclic subgroup $C$, consider its local divisor poset
\[
\mathsf{Div}(C)=\{(C,d): d\mid |C|\},
\]
ordered by $d\mid d'$. 
Define the local cover relation $\prec_C$ on $\mathsf{Div}(C)$ by
\[
(C,d)\prec_C (C,d') \quad \Longleftrightarrow \quad d\mid d' \text{ and there is no } m \text{ with } d < m < d' \text{ and } m\mid |C|.
\]
Note that $(C,d)$ and $(C,d')$ correspond to  unique cliques $K^C_d$ and $K^C_{d'}$ in $EPow(G)$, respectively. Moreover, $(C,d)\prec_C (C,d')$ shows the nested cliques $K^C_d \subset K^C_{d'} \subset K_C$ in the graph.

To identify vertices arising from intersections of maximal cliques, for any two maximal cliques $K_C$ and $K_D$, if $r=|K_C\cap K_D|$, then by Lemma \ref{lem1}, for each divisor $d\mid r$ we identify
\[
(C,d)\sim (D,d).
\]
Let $\sim$ denote the equivalence relation generated by all such identifications, and define
\[
V \;=\; \bigcup_{K_C} \mathsf{Div}(C)\big/\!\sim.
\]

Define the global cover relation $\prec$ on $V$. 
For two classes $[(C,d)]$ and $[(D,d')]$, 
\begin{equation}\label{2}
    [(C,d)] \prec [(D,d')]  \quad \Longleftrightarrow \quad (C,d')\sim (D,d') \text{ and } (C,d)\prec_C (C,d')
\end{equation}

We now prove that the poset $(V,\prec)$ is isomorphic to the cyclic subgroup lattice $\mathcal{L}_c(G)$. Define a map
\[
f: V \longrightarrow \mathcal{C}_G,\qquad f\big([(C,d)]\big)= C_d.
\]
where $C_d$ is a cyclic subgroup of order $d$.

To see that $f$ is well-defined, note that if $[(C,d)]= [(D,d)]$ then  $(C,d)\sim (D,d)$ and $|K_C\cap K_D|$ is divisible by $d$, so both $C$ and $D$ contain a unique subgroup of order $d$, implies $C_d=D_d$.
 The map is injective because distinct classes represent distinct orders or distinct maximal cyclic subgroups that have not been identified. $f$ is surjective because every cyclic subgroup of $G$ appears as $C_d$ for some maximal cyclic subgroup $C$ and some $d\mid |C|$.

It remains to verify  that the cover relation on $V$ corresponds exactly to the cover relation in the cyclic subgroup lattice $\mathcal{L}_c(G)$. 

If $[(C,d)]\prec [(D,d')]$ then by equation (\ref{2}), $(C,d)\prec_c (C,d')$, and hence  $C_d \subset C_{d'}$ with no intermediate cyclic subgroup. 
Thus $f$ preserves cover relations. Conversely, if $H \subset K$ are cyclic subgroups with no intermediate cyclic subgroup, choose a maximal cyclic subgroup $C$ containing $K$. 
Then $H=C_d$ and $K=C_{d'}$ for divisors $d<d'$ of $|C|$ with $(C,d)\prec_c (C,d')$ in $\mathsf{Div}(C)$, so $[(C,d)]\prec [(C,d')]$ in $V$, and $f$ reflects cover relations.

Therefore, $f$ is an isomorphism of posets, and hence $(V,\prec)$ is isomorphic to $\mathcal{L}_c(G)$.
\end{proof}

\begin{remark}
If $d \mid |C|$ and there is no other maximal cyclic subgroup $D$ with 
$|K_C \cap K_D|$ divisible by $d$, then  $[(C,d)]$ is a class of size one. This node never identified 
with a node from another clique. In this case, $(C,d)$ remains as an internal 
vertex in the chain associated to $C$, representing the subgroup $C_d$ that is 
not shared with any other maximal cyclic subgroup of $G$.
\end{remark}

\begin{figure}
   \begin{center}
    \begin{tikzpicture}[
  scale=1.1,
  dot/.style={circle,fill=black,inner sep=1.9pt},
  kA/.style={line width=0.35pt},          
  kB/.style={line width=0.35pt},   
  kC/.style={line width=0.35pt}    
]
  \coordinate (x1) at (0.50, 0.00);
  \coordinate (x2) at (1.10, 0.60);
  \coordinate (x3) at (0.95,-0.55);

  \coordinate (a1) at (-2.00, 0.20);
  \coordinate (a2) at (-1.55,-1.05);
  \coordinate (a3) at (-0.95, 1.05);

  \coordinate (b1) at ( 0.40, 2.25);
  \coordinate (b2) at ( 1.85, 2.55);
  \coordinate (b3) at (-0.50, 2.80);

  \coordinate (c1) at ( 3.20, 0.30);
  \coordinate (c2) at ( 4.40,-1.05);
  \coordinate (c3) at ( 3.60, 1.45);

  \foreach \v in {x1,x2,x3,a1,a2,a3,b1,b2,b3,c1,c2,c3}
    \node[dot] at (\v) {};

  \draw[kA] (x1)--(x2) (x1)--(x3) (x1)--(a1) (x1)--(a2) (x1)--(a3)
            (x2)--(x3) (x2)--(a1) (x2)--(a2) (x2)--(a3)
            (x3)--(a1) (x3)--(a2) (x3)--(a3)
            (a1)--(a2) (a1)--(a3) (a2)--(a3);

  \draw[kB] (x1)--(x2) (x1)--(x3) (x1)--(b1) (x1)--(b2) (x1)--(b3)
            (x2)--(x3) (x2)--(b1) (x2)--(b2) (x2)--(b3)
            (x3)--(b1) (x3)--(b2) (x3)--(b3)
            (b1)--(b2) (b1)--(b3) (b2)--(b3);

  \draw[kC] (x1)--(x2) (x1)--(x3) (x1)--(c1) (x1)--(c2) (x1)--(c3)
            (x2)--(x3) (x2)--(c1) (x2)--(c2) (x2)--(c3)
            (x3)--(c1) (x3)--(c2) (x3)--(c3)
            (c1)--(c2) (c1)--(c3) (c2)--(c3);
\end{tikzpicture}
\caption{The enhanced power graph $C_2\times C_6$}\label{fig1}
\end{center}
\end{figure}
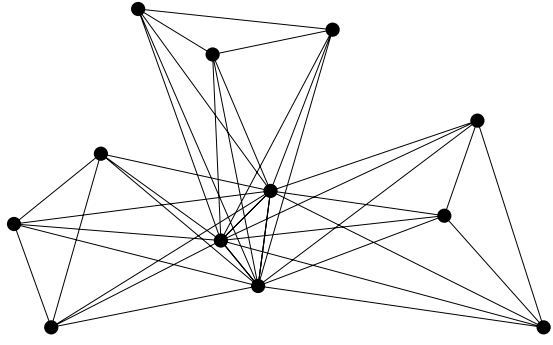

\begin{example}\label{ex1}
Suppose we are given a graph $\Gamma$ with $12$ vertices, as shown in Figure~\ref{fig1}, which is known to be the enhanced power graph of a group, that is, $\Gamma = \mathrm{EPow}(G)$, although no information about $G$ itself is available.  
Our goal is to construct the Hasse diagram of cyclic subgroups $\mathcal{L}_c(G)$, in accordance with the algorithm given in the proof of Theorem \ref{th1} as follow:

 From  $\Gamma$,  we observe three distinct maximal cliques of size six, $K_C$, $K_{D}$ and $K_{E}$, corresponding to the three maximal cyclic subgroups $[(C,6)]$, $[(D,6)]$, $[(E,6)]$ of order six. At this point, we know that $\mathcal{L}_c(G)$  will have three nodes in top level. 

 \begin{enumerate}
     \item For every two maximal cliques $K_X$ and $K_Y$ $(X,Y \in \{C,D,E\})$ we have
    \[
        |K_X \cap K_Y| = 3.
    \]
    Hence, all subgroup vertices of orders dividing three
    in $X$ and $Y$ must coincide.  Therefore we identify:
    \[
        \{(C,3),  (D,3), (E,3)\} = [(C,3)], 
        \qquad 
        \{(C,1),  (D,1), (E,1)\} = [(C,1)].
    \]
            \item   
    Since no intersection of size $2$ appears, 
    the  subgroups of order two,  $[(C,2)],\ [(D,2)],\ [(E,2)]$, remain distinct.   Thus $\mathcal{L}_c(G)$  consists of:
    \[
    \begin{aligned}
      &\text{Level 3: } C_6, D_6, E_6, && |C|=|D|=|E|=6,\\
      &\text{Level 2: } C_3, C_2, D_2, E_2, 
         && |C_3|=3,\ |X_2|=2,\\
      &\text{Level 1: } C_1, && |\{e\}|=1.
    \end{aligned}
    \]
\end{enumerate}

The covering relation are given by equation (\ref{2}):
\[
\begin{aligned}
\{e\} &\prec C_3 \prec C_6,\\
\{e\} &\prec C_3 \prec D_6,\\
\{e\} &\prec C_3 \prec E_6,\\
\{e\} &\prec C_2 \prec C_6,\\
\{e\} &\prec D_2 \prec D_6,\\
\{e\} &\prec E_2 \prec E_6.
\end{aligned}
\]

\noindent
    The cover relations above connect adjacent levels, forming the complete layered structure of $\mathcal{L}_c(G)$ (see Fig.~\ref{fig2}).
\end{example}

\begin{theorem}\label{thm:bottom_up}
The enhanced power graph $\mathrm{EPow}(G)$  can be reconstructed from the lattice of cyclic subgroups,
$\mathcal{L}_c(G)$, without reference to the elements of $G$.
\end{theorem}

\begin{proof}
We are given only the Hasse diagram $\mathcal{L}_c(G)$, with no additional information about the group $G$ itself. 
For each cyclic subgroup $C \le G$, the enhanced power graph contains the complete subgraph $K_C$ on the vertex set $C$. 
We reconstruct the  graph $\Gamma$ by assembling these cliques level by level along the Hasse diagram $\mathcal{L}_c(G)$, beginning from the subgroup $\{e\}$ and proceeding upward.


Suppose that stage $t$ from $\mathcal{L}_c(G)$ where we have processed a collection $\mathcal{L}^{(t)}$ of cyclic subgroups
of $G$. We assume inductively that:

\begin{itemize}
    \item For every $C \in \mathcal{L}^{(t)}$, the clique $K_C$ has already been inserted into $\Gamma$,
          meaning all pairs $\{x,y\}$ with $x,y \in C$ are edges of $\Gamma$.
    \item Whenever two cyclic subgroups $A,B \in \mathcal{L}^{(t)}$,
          $A \cap B = D$, the vertices of $D$ are already identified consistently in $\Gamma$ i.e.\ $D$ is represented as the same set of vertices inside both $A$ and $B$.
\end{itemize}
Now consider the level $\mathcal{L}^{(t+1)}$ consists of all cyclic subgroups $C$ that cover at least
one subgroup we have already processed, and which have not themselves been processed yet.

For each such $C \in \mathcal{L}^{(t+1)}$, do the following:
\begin{enumerate}
\item[(a)] 
Insert all edges of $K_C$ into $\Gamma$, i.e.\ for all $x,y \in C$, add the edge $\{x,y\}$ to $\Gamma$.

\item[(b)] 
Let $D \in \mathcal{L}^{(t)}$ be any subgroup with $D \prec C$ in $\mathcal{L}_c(G)$.  
Since $D$ is already processed, its vertices are already fixed in $\Gamma$.
We now identify $D$
 inside $C$, requiring that the representation of 
$D$
 as a subset of 
$C$
 is represented by the same vertices in $\Gamma$
 as the copy of $D$ we have already constructed in earlier stages.
 This forces $K_C$ to be glued onto the existing graph $\Gamma$ exactly along $K_D$.

If $C$ covers several distinct processed subgroups $D_1, D_2, \dots$ in $\mathcal{L}_c(G)$,
then we glue $K_C$ along each of them simultaneously.  
Since $\mathcal{L}_c(G)$ tells us all $D_i \prec C$, this gluing is canonical.
\end{enumerate}

After performing (a) and (b) for every $C \in \mathcal{L}^{(t+1)}$, the inductive hypotheses
remain true if we replace $\mathcal{L}^{(t)}$ by $\mathcal{L}^{(t+1)}$. Furthermore, every subgroup in $\mathcal{L}^{(t+1)}$ has had its clique inserted,
and intersections are glued consistently.

Since $\mathcal{L}_c(G)$ is finite and each cover relation strictly increases subgroup size,
this process eventually exhausts all cyclic subgroups of $G$.  
That is, for some $N$, we have
\[
   \mathcal{C}_G = \bigcup_{t=0}^N \mathcal{L}^{(t)}.
\]
At that point, every cyclic subgroup $C \le G$ has been inserted as a clique $K_C$ into $H$,
and whenever $C$ contains $D$, the vertices of $D$ inside $C$ are identified with the
already fixed vertices of $D$ from previous stages.

We now show that the resulting graph $\Gamma$ is exactly $\mathrm{EPow}(G)$.

\smallskip

Let $xy$ be an edge in $\mathrm{EPow}(G)$.
Then $\langle x,y\rangle = \langle z\rangle$,
so $x$ and $y$ lie in a common cyclic subgroup $C = \langle z\rangle$.
That cyclic subgroup $C$ is some node of $\mathcal{L}_c(G)$, so at some stage $t$ it entered
$\mathcal{L}^{(t)}$ from $\mathcal{L}^{(t-1)}$ via a cover relation $D \prec C$.
At that moment we added all edges of $K_C$ to $H$.
Therefore $\{x,y\} \in E(\Gamma)$.

Next, we show that every edge of $\Gamma$ is in $\mathrm{EPow}(G)$. 
By construction, the only edges we ever add are edges inside some $K_C$
for a cyclic subgroup $C \le G$.
If $xy$ is an edge of $\Gamma$, then $x,y \in C$ for some cyclic $C$.  
That means $xy$ is an edge in $\mathrm{EPow}(G)$ based on the definition of the enhanced power graph.

Finally, we verify that the vertex set of $\Gamma$ coincides with that of $\mathrm{EPow}(G)$. Clearly, every vertex appearing in $\Gamma$ during the construction is an actual element of $G$.
Conversely, for any element $g\in G$, we have $g\in\langle g\rangle$, 
so $g$ appears in the construction when the cyclic subgroup $\langle g\rangle$ is processed. Thus $V(\Gamma) = G = V(\mathrm{EPow}(G))$. Consequently, $\Gamma$ and $\mathrm{EPow}(G)$ have identical vertex and edge sets, and therefore $\Gamma=\mathrm{EPow}(G)$.
\end{proof}
\begin{figure}
    \centering
\tikzset{sgplattice/.style={inner sep=1pt,norm/.style={red!50!blue},char/.style={blue!50!black},
  lin/.style={black!50}},cnj/.style={black!50,yshift=-2.5pt,left=-1pt of #1,scale=0.5,fill=white}}

\begin{tikzpicture}[scale=1.0,sgplattice]
	\node[char] at (3.12,0) (1) {\gn{C1}{C_1}};
	\node[norm] at (4.12,1.09) (2) {\gn{C2}{C_{2(2)}}};
	\node[norm] at (0.125,1.09) (3) {\gn{C2}{C_{2(1)}}};
	\node[norm] at (6.12,1.09) (4) {\gn{C2}{C_{2(3)}}};
	\node[char] at (2.12,1.09) (5) {\gn{C3}{C_3}};
	\node[norm] at (2.12,2.54) (7) {\gn{C6}{C_{6(2)}}};
	\node[norm] at (0.125,2.54) (8) {\gn{C6}{C_{6(1)}}};
	\node[norm] at (6.12,2.54) (9) {\gn{C6}{C_{6(3)}}};
	\draw[lin] (1)--(2) (1)--(3) (1)--(4) (1)--(5) 
	(2)--(7) (5)--(7) (3)--(8) (5)--(8) (4)--(9) (5)--(9) ;
\end{tikzpicture}
    \caption{ $\mathcal{L}_c(C_2\times C_6)$}        \label{fig2}

\end{figure}
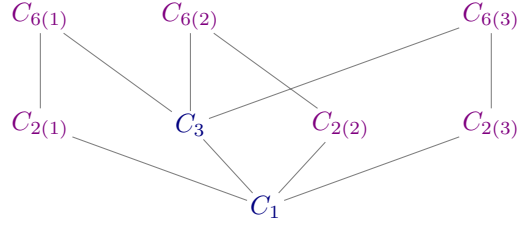

\begin{example}\label{ex2}
    We are given only the Hasse diagram $\mathcal{L}_c(G)$ in Figure~\ref{fig2}, and we do not know the group $G$ itself.  $\mathcal{L}_c(G)$  has three cyclic subgroups of order $2$ (denote them $C_{2(1)},C_{2(2)},C_{2(3)}$), one cyclic subgroup of order $3$ (denote it $C_3$) on the middle level, and three cyclic subgroups of order $6$ (denote them $C_{6(1)},C_{6(2)},C_{6(3)}$) on the top level, together with the trivial subgroup $C_1=\{e\}$.  We reconstruct the  graph $\Gamma$ by assembling these cliques level by level along the Hasse diagram $\mathcal{L}_c(G)$, beginning from the subgroup $\{e\}$ and proceeding upward.

\smallskip
\noindent\text{Stage $t=0$:}
Set $\mathcal{L}^{(0)}=\{\{e\}\}$. Insert the clique $K_{\{e\}}$.

\smallskip
\noindent\text{Stage $t=1$:}
Let $\mathcal{L}^{(1)}=\{\{e\},\, C_3,\, C_{2(1)},C_{2(2)},C_{2(3)}\}$. 
For each $H\in \mathcal{L}^{(1)}$, insert the clique $K_H$ and glue it to $\{e\}$ where this group already processed in the last level and we get the Fig. \ref{fig3}(a).

\smallskip
\noindent\text{Stage $t=2$:}
Let $\mathcal{L}^{(2)}=\{\{e\},\, C_3,\,C_{2(1)},C_{2(2)},C_{2(3)},\, C_{6(1)},C_{6(2)},C_{6(3)}\}$.
For each $C_{6(i)}$ (order $6$), insert the clique $K_{C_{6(i)}}$ and glue it along its already processed
subgroups: the unique subgroup  $C_3$ and the corresponding subgroup of order $2$
(denoted $C_{2(i)}$).
\smallskip
After Stage $t=2$, every cyclic subgroup has been inserted as a clique and all intersections have been glued
consistently along previously processed subgroups (see Fig. \ref{fig3}(b)). The resulting graph $\Gamma$ has:
\begin{enumerate}
    \item $K_{C_{6(1)}},K_{C_{6(2)}},K_{C_{6(3)}}$ of size six,
\item $K_{C_3}$  of size 3 contained in each $K_{C_{6(i)}}$,
\item $K_{C_{2(i)}}$ of size 2 contained in  $K_{C_{6(i)}}$, $1\leq i\leq 3$.
\end{enumerate}
By Theorem~\ref{thm:bottom_up}, the resulting graph $\Gamma$ is exactly $\mathrm{EPow}(G)$, as shown in Fig.~\ref{fig1}.

\end{example}

\begin{figure}
\label{fig3}
   \begin{center}
   \tikzset{char/.style={blue!50!black}}
    \begin{tikzpicture}[
  scale=1.1,
  dot/.style={circle,fill=black,inner sep=1.9pt},
  kA/.style={line width=0.35pt},          
  kB/.style={line width=0.35pt},   
  kC/.style={line width=0.35pt}  
]
 \coordinate (y0) at (-3, 0.00);
  \coordinate (y1) at (-4, 1);
  \coordinate (y2) at (-2, 1);
  \coordinate (y3) at (-5,1);
  \coordinate (y4) at (-1,1);
  \coordinate (y5) at (0,1);
\draw[kC] (y0)--(y1) (y0)--(y2) (y0)--(y3) (y0)--(y4) (y0)--(y5)
            (y1)--(y2);
            \foreach \v in {y0,y1,y2,y3,y4,y5}
    \node[dot] at (\v) {};
    \node[char] at (-3,-1) (1) {\gn{(a)}{(a)}};
  \coordinate (x0) at (3, 0.00);
  \coordinate (x1) at (2.4, 1.3);
  \coordinate (x2) at (3.9, 1.3);
  \coordinate (c1) at (0, 2.5);
  \coordinate (c2) at (1, 2.5);
  \coordinate (d1) at (2.4, 2.5);
  \coordinate (d2) at (3.9, 2.5);
  \coordinate (e1) at (5.5, 2.5);
  \coordinate (e2) at (6.5, 2.5);
  \coordinate (x3) at (1,1);
  \coordinate (x4) at (5,1);
  \coordinate (x5) at (6,1);

\draw[kC] (x0)--(x1) (x0)--(x2) (x0)--(x3) (x0)--(x4) (x0)--(x5)
            (x1)--(x2) 
            (x0)--(c1) (x1)--(c1) (x2)--(c1) (x0)--(c2) (x1)--(c2) (x2)--(c2) (x3)--(c1) (x3)--(c2) (x1)--(x3) (x2)--(x3) (c1)--(c2)
            (x0)--(d1) (x1)--(d1) (x2)--(d1) (x0)--(d2) (x1)--(d2) (x2)--(d2) (x4)--(d1) (x4)--(d2) (x1)--(x4) (x2)--(x4) (d1)--(d2)
            (x0)--(e1) (x1)--(e1) (x2)--(e1) (x0)--(e2) (x1)--(e2) (x2)--(e2) (x5)--(e1) (x5)--(e2) (x1)--(x5) (x2)--(x5) (e1)--(e2)
             
            ;
            \foreach \v in {x0,x1,x2,x3,x4,x5,c1,c2,d1,d2,e1,e2}
    \node[dot] at (\v) {};
        \node[char] at (3,-1) (1) {\gn{(b)}{(b)}};

\end{tikzpicture}
\caption{ Example of reconstruction of the enhanced power graph from cyclic subgroup lattice: (a) Stage $t=1$, (b) Stage $t=2$  }
\end{center}
\end{figure}
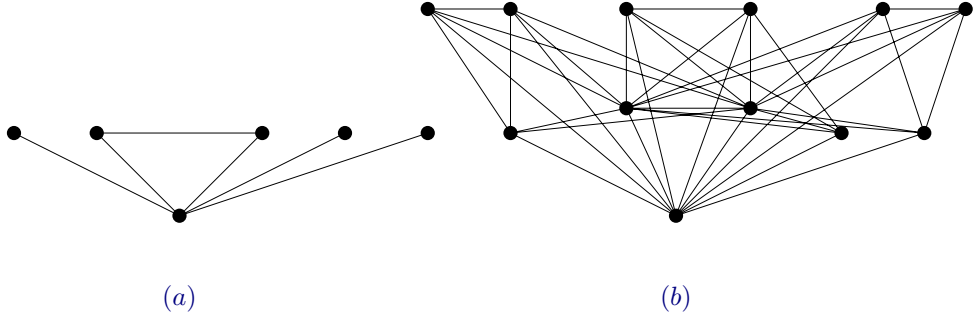

After reconstructing the enhanced power graph in Theorem~\ref{thm:bottom_up}, our next goal is to assign labels to the vertices of the enhanced power graph directly from \(\mathcal{L}_c(G)\), without using any additional structural information about \(G\).

Recall that each vertex of $\mathcal{L}_c(G)$ can be renamed as an equivalence class 
$
[(C,d)] \in (V,\prec)
$
representing the unique cyclic subgroup $C_d \le G$ of order $d$, where 
$(C,d) \sim (C',d)$ whenever $C_d = C'_d$ (cf.\ the identification rule in 
the proof of Theorem \ref{th1}).  
Thus, each class $[(C,d)]$ corresponds to one subgroup of $G$. Note that $[(C,d)]$ is not just depend on maximal subgroup $C$, it
can be the intersection of several maximal subgroups. 
For a cyclic subgroup \(H \le G\) of order \(d\),  define the set of its
immediate predecessors in \(\mathcal{L}_c(G)\) by
\[
\mathrm{Pred}(H) = \{\, E \le H : E \prec H \text{ in } \mathcal{L}_c(G) \,\},
\]
Equivalently, if $d = p_1^{\alpha_1}\cdots p_r^{\alpha_r}$, then 
$
\mathrm{Pred}(H) = \{\, H_{d/p} : p \mid d\,\}$,
where $p$ is prime and $|H_{d/p}|=\frac{d}{p}$.
We now define the set of new vertices introduced when attaching the clique $K_C$ in the proof of
Theorem~\ref{thm:bottom_up} (part (b)).  If $C_d$ covers distinct processed subgroups
$\mathrm{Pred}(C_d)$ in $\mathcal{L}_c(G)$, then $K_C$ is glued to the already existing
vertices corresponding to these subgroups in $\Gamma$.  Hence we set
\[
\operatorname{New}([(C,d)]) 
:= 
C_d \setminus \bigcup_{E\in \mathrm{Pred}(C_d)} E.
\]

\begin{lemma}\label{prop:new-is-generators-global}
For each class $[(C,d)] \in V$, we have
\[
\mathrm{New}([(C,d)]) = \{\, x \in G : \langle x \rangle = C_d \,\}.
\]
Consequently,
$|\mathrm{New}([(C,d)])| = \varphi(d).$

\end{lemma}
\begin{proof}
The predecessors of $C_d$ in $\mathcal{L}_c(G)$ are precisely its maximal proper cyclic
subgroups $H_{d/p}$ with $p\mid d$.  
Every proper subgroup of $C_d$ is contained in one of these $H_{d/p}$.  
Thus
\[
\bigcup_{E\in \mathrm{Pred}(C_d)} E \;=\; 
\bigcup_{p\mid d} C_{d/p}
\;=\; \{\,x\in C_d : |x| < d\,\}.
\]
Therefore
\[
\mathrm{New}([(C,d)]) \;=\; 
C_d \setminus \bigcup_{p\mid d} C_{d/p}
\;=\;
\{\,x\in C_d : |x| = d\,\},
\]
which are exactly the generators of \(C_d\), and therefore
\(|\mathrm{New}([(C,d)])|=\varphi(d)\).
\end{proof}
Now we are ready to label the $EPow(G)$ obtained from $\mathcal{L}_c(G)$ as the following theorem. 
  
   \begin{theorem}\label{thm:canonical-labels}
Every vertex \(x\in \mathrm{EPow}(G)\) admits a canonical label determined by its cyclic subgroup
and a generator index:
\[
x \ \longmapsto\ (C_d,i),
\]
where \(\langle x\rangle = C_d\) is the unique cyclic subgroup generated by \(x\) of order \(d\), and
\(1 \le i \le \varphi(d)\).
\end{theorem}

\begin{proof}
Define a map
\[
f : G \longrightarrow \{(C_d,i) : C_d\in\mathcal{L}_c(G),\, 1\le i \le \varphi(d)\}
\]
as follows: for each \(x\in G\), let \(\langle x\rangle = C_d\).  
Since \(C_d\) is cyclic of order \(d\), it has exactly \(\varphi(d)\) generators.  
Fix an ordering of the generators of \(C_d\) as
\[
\mathrm{New}([(C,d)])=\{g_1,g_2,\dots,g_{\varphi(d)}\},
\]
where \(\mathrm{New}([(C,d)])\) denotes the set of elements of order \(d\) in \(C_d\).
Then define \(f(x)=(C_d,i)\), where \(x=g_i\).

The map \(f\) is well-defined because each vertex \(x\in G\) generates a unique cyclic subgroup \(C_d\), 
and appears exactly at the stage when \(C_d\) is introduced in the reconstruction of \(\mathrm{EPow}(G)\) (proof of Theorem~\ref{thm:bottom_up} part (b)).
Moreover, by Lemma~\ref{prop:new-is-generators-global},
\[
|\mathrm{New}([(C,d)])|=\varphi(d),
\]
so each generator \(g_i\) corresponds to exactly one index \(i\).

The map is injective: if \(f(x)=f(y)\), then both \(x\) and \(y\) are the same \(i\)-th generator of the same subgroup \(C_d\), hence \(x=y\).
The map is surjective: given any pair \((C_d,i)\), the element \(g_i\in \mathrm{New}([(C,d)])\) satisfies \(f(g_i)=(C_d,i)\).
Therefore, \(f\) is a bijection, which establishes the stated canonical labeling of the vertices of \(\mathrm{EPow}(G)\).
\end{proof}

\begin{remark}
Because the identification $(C,d)\sim (C',d)$ is performed before any labelling,
no subgroup is represented more than once, and no generator can receive conflicting labels.
The labelling procedure therefore depends only on the structure of $\mathcal{L}_c(G)$,
not on any choice of maximal cyclic subgroup containing $C_d$.
\end{remark}

\begin{example}\label{ex3}
    Recall that in Example~\ref{ex1} we began with unlabeled $\mathrm{EPow}(G)$ and reconstructed the $\mathcal{L}_c(G)$.
Then, in Example~\ref{ex2}, we started from $\mathcal{L}_c(G)$ and reconstructed unlabeled $\mathrm{EPow}(G)$.
We now proceed to assign labels to the   $\mathrm{EPow}(G)$ obtained in Example~\ref{ex2}, by method in Theorem \ref{thm:canonical-labels}.
\begin{enumerate}
    \item \text{Stage \(\mathcal{L}^{(0)}\):}  
    For the trivial subgroup \(\{e\}\), we assign the label \((C_1,1)\).

    \item \text{Stage \(\mathcal{L}^{(1)}\):}  
    The three subgroups of order \(2\), denoted \(C_{2(1)}, C_{2(2)}, C_{2(3)}\), each have a single generator.  
    Thus, their labels are
    \[
    (C_{2(1)},1),\quad (C_{2(2)},1),\quad (C_{2(3)},1).
    \]
    The subgroup of order \(3\), denoted \(C_3\), has two generators.  
    We label them
    \[
    (C_3,1),\quad (C_3,2).
    \]

    \item Stage \(\mathcal{L}^{(2)}\):
    Each subgroup of order \(6\) has \(\varphi(6)=2\) generators.  
    Therefore  their generator labels are :
    \[
    C_{6(1)}:\ (C_{6(1)},1),\ (C_{6(1)},2),
    \]
    \[
    C_{6(2)}:\ (C_{6(2)},1),\ (C_{6(2)},2),
    \]
    \[
    C_{6(3)}:\ (C_{6(3)},1),\ (C_{6(3)},2).
    \]
\end{enumerate}
\end{example}

\section{Reconstructing the power graph and directed power graph from $\mathcal{L}_c(G)$}

Recall that each node of $\mathcal{L}_c(G)$ is an equivalence class $[(C,d)]$
representing the unique cyclic subgroup $C_d\le G$ of order $d$, and that
\[
\mathrm{New}([(C,d)])
\;=\;
\{x\in C_d:\ |x|=d\}
\]
is the set of generators of $C_d$. The  proof of restructuring power graph algorithm is almost similar to argument of Theorem \ref{thm:bottom_up} where we present it as following. 

\begin{theorem}\label{thm:power-bottom-up}
The undirected power graph $\mathrm{Pow}(G)$ can be reconstructed 
from the cyclic subgroup lattice $\mathcal{L}_c(G)$.
\end{theorem}
\begin{proof}
Using the following inductive construction, we build a graph $\Gamma$, and then we show that $\Gamma \cong \mathrm{Pow}(G)$.

\begin{enumerate}
\item Begin with the class $[(\{e\},1)]$, which corresponds to the first vertex in the graph $\Gamma$.

\item Suppose that every class $[(E,e)]$  in $\mathcal{L}_c(G)$ until level $\mathcal{L}^{(t)}$ has already been processed. Let $C_d \in \mathcal{L}^{(t+1)}$ and $\mathrm{Pred}(C_d)=\{E_1,\dots,E_r\}$ be the immediate predecessors of $C_d$ in $\mathcal{L}_c(G)$
and set $\mathrm{New}([(C,d)])$. Insert edges according to the following two rules:

\begin{itemize}
  \item[(E1)] Make a clique $K_{\varphi(d)}$, connect every two vertices in $\mathrm{New}([(C,d)])$.
  \item[(E2)] For every subgroup $E\le C_d$ that is already processed
              (in particular each $E_i\in \mathrm{Pred}(C_d)$, and any lower subgroup), connect every vertex of $\mathrm{New}([(C,d)])$ to every vertex  of $E$.
\end{itemize}
Mark $[(C,d)]$ as processed and continue until all classes are processed.
\end{enumerate}
We now show that the resulting undirected graph $\Gamma$ is exactly the power graph $\mathrm{Pow}(G)$.

 Let $xy$ be an edge in $\Gamma$, $x\in \mathrm{New}([(C,d)])$ and  $y\in E\le C_d$ be any already processed vertex.
Then $\langle y\rangle\le \langle x\rangle$, so $\{x,y\}$ must be an edge in $\mathrm{Pow}(G)$; this is exactly the edge inserted by rule (E2). If $x,x'\in \mathrm{New}([(C,d)])$, then $\langle x\rangle=\langle x'\rangle=C_d$, hence $x,x'$ are adjacent in $\mathrm{Pow}(G)$; this is precisely the adjacency introduced by rule (E1).
Finally, no additional edges are added by $(E_1)$ and $(E_2)$. In particular, if $y\in E_i$ and $z\in E_j$ for two distinct predecessors
$E_i\neq E_j$, then $E_i$ and $E_j$ are incomparable in $\mathcal{L}_c(G)$, so $\langle y\rangle$ and
$\langle z\rangle$ are incomparable and $\{y,z\}$ is not an edge of $\mathrm{Pow}(G)$; Our construction adds no such edge.

Now, take any edge $\{u,v\}$ of $\mathrm{Pow}(G)$. Then $\langle u\rangle\subseteq \langle v\rangle$
or the reverse. Without loss, $\langle u\rangle\le \langle v\rangle$. Let $[(C,d)]$ be the class of $\langle v\rangle$. When $[(C,d)]$ is processed, $v\in \mathrm{New}([(C,d)])$,
and $u$ lies in a subgroup $E\le C_d$ that has already been processed (all proper subgroups are processed earlier).
Hence (E2) inserts the edge $\{u,v\}$. If $\langle u\rangle=\langle v\rangle$, both are in the same $\mathrm{New}([(C,d)])$ and (E1) inserts the edge.
Hence all edges of $\mathrm{Pow}(G)$ appear in $\Gamma$, and consequently $\Gamma = \mathrm{Pow}(G)$.
\end{proof}

\begin{remark}
In the enhanced power graph, when a cyclic subgroup $C_d$ covers several $E_i$, one glues the complete
clique on $C_d$ along all the $E_i$ simultaneously.
If we were to do this in constructing the
power graph, it would incorrectly create edges between generators that lie in distinct, incomparable subgroups $E_i$.
 The rule (E2) prevents such extra edges, it connects the new generators only to vertices in
lower (comparable) subgroups, never across incomparable predecessors. Consequently, a cyclic subgroup induces a
complete subgraph in $\mathrm{Pow}(G)$ if and only if $d$ is a prime power (that is, when its subgroup lattice is a chain).
\end{remark}

\medskip

\noindent

We record the parallel directed construction.

\begin{theorem}\label{thm:dir-power-bottom-up}
 The directed power graph $\overrightarrow{Pow}(G)$ can be reconstructed 
from the cyclic subgroup lattice $\mathcal{L}_c(G)$.
\end{theorem}

\begin{proof}
With the same notation as in Theorem \ref{thm:power-bottom-up}, replacing (E1)–(E2) with:

\begin{itemize}
  \item[(D1)] {Inside $\mathrm{New}([(C,d)])$:} for any $x,x'\in \mathrm{New}([(C,d)])$, add both arcs $x\to x'$ and $x'\to x$.
  \item[(D2)] {Downward orientation:} for every $E\le C_d$ already processed and every
              $x\in \mathrm{New}([(C,d)])$, $y\in E$, add the arc $x\to y$.
\end{itemize}
No arcs are added between vertices belonging to incomparable classes.
The resulting digraph is the directed power graph.
If $x\in \mathrm{New}([(C,d)])$ and $y\in E\le C_d$, then $\langle y\rangle\subseteq \langle x\rangle$,
so $y=x^m$ for some $m\ge 2$, hence $x\to y$ is a directed power edge (inserted by (D2)).
If $x,x'\in \mathrm{New}([(C,d)])$, both generate $C_d$ and are mutual powers, giving the bidirected pair by (D1).
No arc is added across incomparable classes, matching the definition.
\end{proof}


    \begin{example}\label{ex4}
The vertex labels in the directed power graph coincide with those in the enhanced power graph.  
In the first stage, we place the identity element.  
In the second stage, as in Figure~\ref{fig4}, we have the vertices corresponding to the elements of order $2$ and order $3$, with directed edges from each vertex in $(C_{2(i)},1)$ and $(C_{3},j)$ to the identity where $1\leq i \leq 3$ and $1\leq j \leq 2$.  

In the next stage, the vertices corresponding to the elements of order $6$ appear, and we draw directed edges from each $(C_{6(i)},j)$ to the corresponding subgroups $(C_{2(i)},1)$ and $(C_{3},j)$ contained in it.  
The final diagram obtained in this construction is the directed power graph, as illustrated in Figure~\ref{fig4}.

It is important to note that in this case there are no directed edges from $(C_{3},j)$ to $(C_{2(i)},1)$, which distinguishes the directed power graph from the enhanced power graph.
\end{example}

\tikzset{
  midarrow/.style={
    postaction={decorate},
    decoration={markings, mark=at position 0.5 with {\arrow{>}}}
  },
  dot/.style={circle,fill=black,inner sep=1.9pt}
}

\begin{figure}
\label{fig4}
\begin{center}
\begin{tikzpicture}[scale=1.1]

  \coordinate (x0) at (0.00, 0.00);
  \coordinate (x1) at (-.6, 1.3);
  \coordinate (x2) at (.9, 1.3);
  \coordinate (c1) at (-3, 2.5);
  \coordinate (c2) at (-2, 2.5);
  \coordinate (d1) at (-.6, 2.5);
  \coordinate (d2) at (.9, 2.5);
  \coordinate (e1) at (2.5, 2.5);
  \coordinate (e2) at (3.5, 2.5);
  \coordinate (x3) at (-2,1);
  \coordinate (x4) at (2,1);
  \coordinate (x5) at (3,1);

  \draw[midarrow] (x1) -- (x0);
  \draw[midarrow] (x2) -- (x0);
  \draw[midarrow] (x3) -- (x0);
  \draw[midarrow] (x4) -- (x0);
  \draw[midarrow] (x5) -- (x0);

  \draw[midarrow] (x1) -- (x2);
  \draw[midarrow] (x2) -- (x1);

  \foreach \a in {x0,x1,x2,x3} {
    \draw[midarrow] (c1) -- (\a);
    \draw[midarrow] (c2) -- (\a);
  }
    \draw[midarrow] (c1) -- (c2);
  \draw[midarrow] (c2) -- (c1);

  \foreach \a in {x0,x1,x2,x4} {
    \draw[midarrow] (d1) -- (\a);
    \draw[midarrow] (d2) -- (\a);
  }
  
  \draw[midarrow] (d1) -- (d2);
  \draw[midarrow] (d2) -- (d1);

  \foreach \a in {x0,x1,x2,x5} {
    \draw[midarrow] (e1) -- (\a);
    \draw[midarrow] (e2) -- (\a);
  }
  
  \draw[midarrow] (e1) -- (e2);
  \draw[midarrow] (e2) -- (e1);

  \foreach \v in {x0,x1,x2,x3,x4,x5,c1,c2,d1,d2,e1,e2}
     \node[dot] at (\v) {};

\end{tikzpicture}
\caption{The directed power graph}
\end{center}
\end{figure}
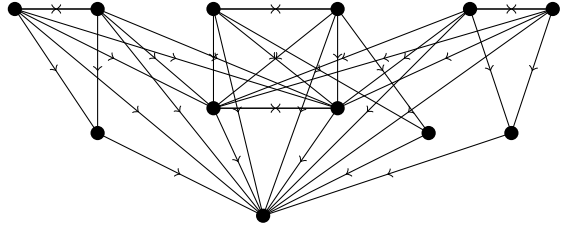

\begin{corollary}
\label{cor:diff}
The difference graph $\mathrm{D}(G)$ can be reconstructed 
from the cyclic subgroup lattice $\mathcal{L}_c(G)$.
\end{corollary}

\begin{proof}
By the reconstruction theorems established above, both $\mathrm{EPow}(G)$ and
$\mathrm{Pow}(G)$ are obtained  from $\mathcal{L}_c(G)$ (no reference
to elements is required). Hence their edge sets are determined by
$\mathcal{L}_c(G)$, and so is the set-theoretic difference
$E(\mathrm{EPow}(G))\setminus E(\mathrm{Pow}(G))$.

Equivalently, using only $\mathcal{L}_c(G)$, an unordered pair $\{x,y\}$ lies in
$E(\mathrm{D}(G))$ iff there exists a cyclic subgroup $C$ with $\{x,y\}\subseteq C$
(\,detectable in $\mathcal{L}_c(G)$\,) and the nodes corresponding to
$\langle x\rangle$ and $\langle y\rangle$ are incomparable in the chain
of subgroups of $C$ (so $\{x,y\}$ is an EPow-edge but not a Pow-edge).
Thus $\mathrm{D}(G)$ depends only on $\mathcal{L}_c(G)$.
\end{proof}

\noindent\textbf{Proof of the Main Theorem.} 
Theorems~\ref{th1},~\ref{thm:bottom_up}, and~\ref{thm:canonical-labels} together establish the mutual reconstruction 
between the enhanced power graph $\mathrm{EPow}(G)$ and the cyclic subgroup lattice $\mathcal{L}_c(G)$ of a finite group $G$.  
Furthermore, Theorems~\ref{thm:power-bottom-up} and~\ref{thm:dir-power-bottom-up} demonstrate that 
both the power graph and the directed power graph of $G$ 
can also be reconstructed directly from the cyclic subgroup lattice.  Corollary~\ref{cor:diff} proved the construction of a difference graph from $\mathcal{L}_c(G)$.
This completes the proof of the Main Theorem.

\section{The Isomorphism Problem for Power-Type Graphs}

It is well known that isomorphic power-type graphs do not necessarily correspond to isomorphic groups. 
For example, among the $14$ groups of order $16$, there exist only $12$ distinct power graphs. 
  We say that a group $G_1$ is \emph{determined  by its power graph} if, for every group $G_2$ such that
$\operatorname{Pow}(G_1) \cong \operatorname{Pow}(G_2)$ then $G_1 \cong G_2$. Thus, the classical question may be stated as follows:
 Which groups are determined by their power-type graphs? 

We denote by $\mathcal{L}(G)$ the (complete) subgroup lattice of a group. A group $G$ is said to be determined by its subgroup lattice if it is isomorphic to every group $H$ such that $\mathcal{L}(G)\cong \mathcal{L}(H)$. We also say that $G$ and $G'$ are $\mathcal{L}$-isomorphic if $\mathcal{L}(G)\cong \mathcal{L}(G')$. Similarly, for the sublattice \(\mathcal{L}_c(G)\), we use the terms "determined by its cyclic subgroup lattice" and "\(\mathcal{L}_c\)-isomorphic". From our main theorem, we have:

\begin{corollary}\label{th:hasse_equiv}
A group is determined by its power-type graph if and only if it is determined by its lattice of cyclic subgroups.
\end{corollary}

 Moreover, when combined Corollary \ref{th:hasse_equiv}, by Corollary~7 in~\cite{zahirovic2020study}, we have: 
\begin{corollary}\label{cor4}
For groups $G_1$ and $G_2$, the following statements are equivalent:  
\begin{enumerate}
    \item the lattice of cyclic subgroups of $G_1$ and $G_2$  are isomorphic.
    \item the directed power graphs of $G_1$ and $G_2$ are isomorphic.
   \item the enhanced power graphs of $G_1$ and $G_2$ are isomorphic.
    \item the power graphs of $G_1$ and $G_2$ are isomorphic.
\end{enumerate}
\end{corollary}

We consider the connection between the isomorphism problem for subgroup lattices and that for power-type graphs from two perspectives:

\begin{enumerate}
    \item If two groups are \(\mathcal{L}\)-isomorphic, then they are also \(\mathcal{L}_c\)-isomorphic, and by Corollary~\ref{cor4}, their power-type graphs are isomorphic. Therefore, the isomorphism problem for the power-type graph reduces to the isomorphism problem for the lattice of cyclic subgroups.  Thus, instead of comparing the full graph structures on the elements of the groups, it is sufficient to compare the simpler poset $\mathcal{L}_c(G)$. For example, by Corollary~12 in \cite{tuarnuauceanu2006groups}, two finite abelian groups are isomorphic if and only if they are \(\mathcal{L}\)-isomorphic and have the same order. 
Thus, for abelian groups of the same order, isomorphism of their subgroup lattices directly implies isomorphism of their power graphs (a fact also proved in \cite{cameron2010power}  using the group structure). Moreover, if two groups are \(\mathcal{L}\)-isomorphic, then they have the same number of elements of each order (by Corollary~3 of \cite{cameron2011power}).

\item If a group \(G\) is determined by its power-type graph, then \(G\) is determined by its cyclic subgroup lattice, and consequently it is also determined by its subgroup lattice. Therefore, any group that is determined by its power-type graph is also determined by its subgroup lattice. Some examples of groups that are known to be determined by their power-type graph (see \cite{mirzargar2025finite} and references there in), and hence also determined by their subgroup lattice, include:
 generalized quaternion groups, symmetric groups, dihedral groups, \(Q_8 \times \mathbb{Z}_n\) with \(n\) odd, \(\left(\prod_{i=1}^{m} \mathbb{Z}_2\right) \times \mathbb{Z}_n\) with \(n\) odd, and \(\mathbb{Z}_p \times \mathbb{Z}_p \times \mathbb{Z}_n\) where \(\gcd(n,p)=1\),
the semidihedral group $SD_{2^n}=\langle x, a\mid x^2 = 1 = a^{2^{n-1}}, a^x = a^{2^{n-2}-1} \rangle$, where $n \geq 3$,  the  group $\langle x, a \mid x^p = 1 = a^{p^{n-1}}, a^x = a^{1+p^{n-2}}\rangle$, where $n \geq 3$.

\end{enumerate}

Moreover, in results such as Theorem~2.8 of \cite{mirzargar2025finite}, it is shown that  $EPow(G)\cong EPow(\mathbb{Z}_p\times \mathbb{Z}_p\times \mathbb{Z}_p\times \mathbb{Z}_n)$  if and only if $G\cong\mathbb{Z}_p\times\mathbb{Z}_p \times\mathbb{Z}_p\times \mathbb{Z}_n$ or $K\times\mathbb{Z}_n$, where $K$ represents the unique non-abelian group of order $p^3$ with exponent $p$ and $gcd(n,p)=1$. It is straightforward to verify that $\mathbb{Z}_p\times\mathbb{Z}_p \times\mathbb{Z}_p\times \mathbb{Z}_n$ and $K\times\mathbb{Z}_n$ ($K$ is the non-abelian group of order $p^3$ with exponent $p$ and $gcd(n,p)=1$) are \(\mathcal{L}_c\)-isomorphic
 
Also, \cite{mirzargar2022finite} classifies power graphs according to group orders, and in particular identifies orders for which every group of that order is determined by its power graph. By the above discussion, the same classification applies to groups determined by their subgroup lattice. For instance, by Proposition~2.1 of \cite{mirzargar2022finite}, every group of order \(2p^2\)($p$ is prime), is determined by its subgroup lattice.

\section{Future Work}

The results of this paper establish a precise two way correspondence between the enhanced power graph 
and the cyclic subgroup lattice of a finite group, and show that all power-type graphs, that is, power, directed power, 
enhanced, and difference graphs, can be reconstructed solely from $\mathcal{L}_c(G)$. 
This framework opens several promising directions for further research. 
In particular, the survey of Cameron~\cite{cameron2022} highlights a number of open questions concerning 
the structural role and classification of power-type graphs; our reconstruction approach provides a natural 
setting in which several of these questions can now be revisited.

A natural next step is to investigate the extent to which the reconstruction procedure can be extended 
beyond cyclic subgroups. Since every subgroup is generated by its collection of cyclic subgroups, 
it is reasonable to ask whether the full subgroup lattice $\mathcal{L}(G)$ can be recovered from 
the enhanced power graph or from other suitably enriched graph invariants. 
Addressing this would require understanding how non-cyclic subgroups can be recognized as joins or 
intersections of families of cyclic subgroups appearing in the reconstructed lattice.

Another line of inquiry concerns the classification of groups determined by their power-type graphs. 
While our reconstruction theorem resolves the isomorphism problem for enhanced power graphs in terms of 
cyclic subgroup lattices, identifying exactly which finite groups are determined by their full subgroup 
lattices remains a rich and subtle question. Exploring the implications of the present framework for these 
classical classification problems may lead to new characterizations within specific group families, 
such as $p$-groups or nilpotent groups.

Finally, the reconstruction viewpoint suggests potential applications in computational group theory. 
The ability to recover structural information from graph data raises the possibility of efficient 
algorithms for identifying group properties or distinguishing non-isomorphic groups using only 
graphical representations. Developing such algorithms, perhaps in analogy with those studied in 
\cite{das2023isomorphismproblempowergraphs}, presents an intriguing and promising direction for future work.

\section*{Acknowledgment}
 We wish to express our sincere gratitude to the Visiting or Sabbatical Scientist Support Program for their invaluable assistance. Furthermore, the authors acknowledges the financial support provided by the Scientific and Technological Research Council of Turkey (TUBITAK) through the Bideb 2221 program.

\bibliography{ref}
\bibliographystyle{abbrv}

\end{document}